\documentclass{article}

\usepackage{arxiv}

\usepackage[utf8]{inputenc} 
\usepackage[T1]{fontenc}    
\usepackage{hyperref}       
\usepackage{url}            
\usepackage{booktabs}       
\usepackage{amsfonts}       
\usepackage{nicefrac}       
\usepackage{microtype}      
\usepackage{lipsum}
\usepackage{graphicx}
\usepackage{amsmath,amssymb,amsthm,mathtools}
\usepackage{enumitem}
\usepackage{subcaption}
\usepackage{xcolor}
\newtheorem{theorem}{Theorem}[section]
\newtheorem{lemma}[theorem]{Lemma}
\newtheorem{proposition}[theorem]{Proposition}
\newtheorem{corollary}[theorem]{Corollary}

\newtheorem{remark}[theorem]{Remark}

\newtheorem{conjecture}[theorem]{Conjecture}

\newtheorem{claim}[theorem]{Claim}
\usepackage[numbers]{natbib}

\graphicspath{ {./images/} }

\title{Two conjectures on graphs and their edge-path matrices}

\author{
 Metrose Metsidik \\
  College of Mathematical Sciences\\
  Xinjiang Normal University\\
  Urumqi 830017, China \\
  \texttt{metrose@xjnu.edu.cn} \\
   \And
 Xian'an Jin \\
  School of Mathematical Sciences\\
  Xiamen University\\
  Xiamen 361005, China \\
  \texttt{xajin@xmu.edu.cn} \\
}

\begin{document}
\maketitle
\begin{abstract}
The edge-path matrix is a square matrix where each off-diagonal entry records the maximum number of edge-disjoint paths between the corresponding pair of vertices. Akbari et al.~[On edge-path eigenvalues of graphs, {\it Linear Multilinear Algebra} 70 (2022) 2998–3008] proposed two conjectures: Conjecture 1 relates the edge-path matrix to an upper bound on the number of edges in the graph, while Conjecture 2 asserts that a graph is Eulerian if and only if all entries of its edge-path matrix are even. In this paper, we prove the two conjectures.
\end{abstract}
\section{Introduction}

The path matrix \( P(G) \) of a graph \( G \) with \( n \) vertices is defined as an \( n \times n \) matrix whose \((i,j)\)-entry represents the maximum number of internally vertex-disjoint paths between vertices \( v_i \) and \( v_j \) for \( i \neq j \), with zero entries on the diagonal. Fundamental properties of this matrix and its spectral characteristics have been extensively studied in~\cite{Akbari1,Patekar,Shikare}.

An edge-analogue of this concept, called the edge-path matrix \( EP(G) \), was introduced by Akbari et al.~\cite{Akbari}. This matrix is defined as \( EP(G) = (p_{ij})_{n \times n} \), where each off-diagonal entry \( p_{ij} \) records the maximum number of edge-disjoint paths between vertices \( v_i \) and \( v_j \), while diagonal entries are zero. It is straightforward to observe that for many fundamental graph classes — including trees, cycles, and complete graphs — the edge-path matrix \( EP(G) \) equals a scalar multiple of \( J - I \), where $I$ denotes the identity matrix and $J$  represents the square matrix with all entries equal to 1. Obviously, a graph \( G \) is a tree if and only if \( EP(G) = J - I \). Akbari et al.~\cite{Akbari} conducted a systematic study of graphs \( G \) for which \( EP(G) \) takes this form, specifically characterizing all graphs satisfying \( EP(G) = 2(J - I) \). In their work, they subsequently formulated the following two conjectures concerning the relationship between a graph's structure and its edge-path matrix.

\begin{conjecture}\cite{Akbari}\label{C-1}
Let $G$ be a simple graph of order $n$ and $EP(G) \leq q(J-I)$. Then 
$$|E(G)|\leq 
    (q+1)\frac{n-1}{2}.$$
\end{conjecture}

\begin{conjecture}\cite{Akbari}\label{C-2}
A simple graph is Eulerian if and only if every entry of its edge-path matrix is even. 
\end{conjecture}

In this paper, we first identify a class of graphs that attain the maximum possible number of edges among all graphs whose edge-path matrix \(EP(G)\) is a scalar multiple of \(J - I\). This characterization then leads naturally to a resolution of Conjecture~\ref{C-1}. Finally, we provide a complete proof of Conjecture~\ref{C-2} using Euler's characterization of Eulerian graphs together with the inclusion–exclusion principle.

\section{Graphs with edge-path matrix $q(J-I)$ }

In this section, we verify Conjecture~\ref{C-1}.

A {\it block} of \( G \) is a maximal connected subgraph without any cut-vertices. It is well known that the blocks of a nontrivial tree are all isomorphic to the complete graph \( K_2 \), and for arbitrary connected graphs, blocks form a tree-like intersection structure. A 2-{\it cactus} is a connected graph in which every block is a cycle. The following theorem provides a characterization of 2-cactus graphs in terms of their edge-path matrices.

\begin{theorem}\cite{Akbari}\label{T-1}
Let $G$ be a simple graph of order $n$. Then $EP(G) = 2(J_n-I_n)$ if and only if
$G$ is a 2-cactus.
\end{theorem}

In~\cite{Akbari}, the authors established the validity of Conjecture~\ref{C-1} for the case $q=2$. Below we present an alternative and simpler proof.

\begin{proposition}\label{MP-1}
Let $G$ be a simple graph of order $n$ and $EP(G) = 2(J_n-I_n)$. Then $$n\leq |E(G)|\leq 3\times\frac{n-1}{2}.$$
\end{proposition}
\begin{proof}
By Theorem~\ref{T-1}, every block of $G$ is a cycle. Consequently, by deleting one edge from each block, we can reduce $G$ to a tree. Let $s$ denote the number of blocks in $G$. Then the edge count satisfies $|E(G)| = n - 1 + s$, demonstrating that the total number of edges in $G$ grows linearly with the number of blocks.  

The minimal edge count occurs when $G$ consists of a single block (i.e., when $G$ is itself a cycle). On the other hand, since the smallest possible block is a 3-cycle, the maximum number of blocks in $G$ is bounded above by $1 + \left\lfloor \frac{n - 3}{2} \right\rfloor$.  

We observe the following structural properties of such graph $G$ with maximum number of blocks:

For odd $n$, every block in $G$ is a 3-cycle.

For even $n$, exactly one block must be a 4-cycle, while all other $ \frac{n-4}{2} $ blocks remain 3-cycles.

Consequently, we can bound the number of edges as follows:
\[
|E(G)| \leq 3 \times \left(1 + \left\lfloor \frac{n - 3}{2} \right\rfloor + \left(\frac{n - 3}{2} - \left\lfloor \frac{n - 3}{2} \right\rfloor\right)\right) = 3 \times \frac{n-1}{2}.
\]
\end{proof}

A graph is called a $k$-{\it cactus} if it is connected and every block has an edge-path matrix equal to $k(J-I)$. Since every 2-connected graph admits an ear decomposition starting with any cycle as the initial cycle, every block of a 2-cactus must also admit such a decomposition. However, if a block contains a cycle with an attached ear, there must exist at least three edge-disjoint paths between the two vertices shared by the cycle and the ear. This contradicts the structure of a 2-cactus, where no three edge-disjoint paths are allowed between any two vertices. Therefore, every block of a 2-cactus must itself be a simple cycle. This indicates that the 2-cactus defined above is a special case of the $k$-cactus when $k=2$. The following lemma establishes a characterization of $k$-cactus graphs based on their edge-path matrices.

\begin{lemma}\label{ML-0}
Let $G$ be a graph of order $n$. Then $EP(G) = q(J_n-I_n)$ if and only if
$G$ is a $q$-cactus graph.
\end{lemma}
\begin{proof}
The necessity is obvious. We prove the sufficiency by induction on the number of blocks in the graph \( G \).  

{\bf Base case:} If \( G \) consists of a single block, the result holds trivially.  

{\bf Inductive step:} Assume the claim holds for all \( q \)-cactus graphs with fewer blocks than \( G \). 

Recall that any two distinct blocks share at most one cut-vertex of $G$. A {\it leaf block} of a graph \( G \) is a block that contains exactly one cut-vertex of $G$.

Since \( G \) has at least two blocks, it contains at least two leaf blocks (a well-known property of block graphs). Let \( B \) be a leaf block of \( G \), and let \( x \) be the unique cut-vertex of \( G \) in \( B \).  

By the induction hypothesis, the subgraph \( G - (V(B) \setminus \{x\}) \) satisfies  
\[  
EP(G - (V(B) \setminus \{x\})) = q(J_{n-t+1} - I_{n-t+1}),  
\]  
where \( t = |V(B)| \). Similarly, the leaf block \( B \) satisfies  
\[  
EP(B) = q(J_t - I_t).  
\]  

It remains to verify the case where \( y \in V(G) \setminus V(B) \) and \( z \in V(B) \setminus \{x\} \). By the leaf block structure, every \( y \)-\( z \) path must pass through \( x \). Since $EP(G - (V(B) \setminus \{x\})) = q(J_{n-t+1} - I_{n-t+1})$ and $EP(B) = q(J_t - I_t)$, there are exactly \( q \) edge-disjoint paths between:  

 \( y \) and \( x \) in \( G - (V(B) \setminus \{x\}) \), and  

\( z \) and \( x \) in \( B \).  

Concatenating these paths at \( x \) yields exactly \( q \) edge-disjoint \( y \)-\( z \) paths in \( G \). This completes the inductive argument.  

\end{proof}

The following two technical terms are frequently used in this paper. Let \(\operatorname{Cir}(n, S)\) be a circulant graph with connection set \(S\subseteq\{1,\dots,n-1\}\) satisfying
\[
s\in S \iff n-s\in S. 
\]
It is well known that circulant graphs are vertex-transitive, and that the edge connectivity of a connected vertex-transitive graph equals its valency. Thus, if the connection set \(S\) of \(\operatorname{Cir}(n,S)\) contains an integer \(i\) coprime to \(n\), then the subgraph \(\operatorname{Cir}(n,\{i,n-i\})\) of \(\operatorname{Cir}(n,S)\) is a Hamilton cycle. It follows that \(\operatorname{Cir}(n,S)\) is connected, has edge connectivity \(|S|\), and is therefore a \(|S|\)-cactus.

Let \(l\) be the largest integer at most \(\lfloor n/2\rfloor\) that is coprime to \(n\). Then
\[
l = 
\begin{cases}
(n-1)/2, & \text{if } n \text{ is odd},\\ 
n/2 - 1, & \text{if } n \text{ is even and } n \equiv 0 \pmod{4},\\ 
n/2 - 2, & \text{if } n \text{ is even and } n \equiv 2 \pmod{4}.
\end{cases}
\]

To verify Conjecture~\ref{C-1}, we first prove the following two lemmas.

\begin{lemma}\label{ML-1001}
Let \(G=\operatorname{Cir}(n,S)\), where \(|S|=q\ge 4\) and \(l\in S\). Suppose \(F\subseteq E(G)\) is a matching with \(|F|<\lfloor n/2\rfloor\). Then, for every pair of vertices \(u,v\notin V(F)\), the graph \(G-F\) contains \(q\) edge-disjoint \(u\)--\(v\) paths.
\end{lemma}
\begin{proof}
Let \(U = V(G) \setminus V(F)\) be the set of vertices not covered by \(F\). Since \(|F| < \lfloor n/2 \rfloor\), we have \(|U| \ge 2\), and every vertex \(u \in U\) has degree \(q\) in \(G-F\). We prove in three steps that for any two distinct vertices \(u, v \in U\), there exist \(q\) edge-disjoint \(u\)-\(v\) paths in \(G-F\).

\textbf{Step 1: Finding an isoperimetric lower bound.}

Let \(S^+ = S \cap \{1,2,\dots,\lfloor n/2\rfloor\}\), and let \(s_1\) be the smallest element of \(S^+\) different from \(l\). Since \(|S|\ge 4\), we have \(|S^+|\ge 2\). If \(n/2 \notin S\), then \(q = 2|S^+|\); if \(n/2 \in S\), then \(q = 2|S^+|-1\). For \(X \subseteq V\), let \(\partial X\) denote the set of edges joining \(X\) to \(\overline{X}\).

For each step \(s\in S^+\), define
\[
\partial_s X = \bigl\{\{i,i+s\}: \text{exactly one of } i \text{ and } i+s \text{ belongs to } X\bigr\}.
\]
Then we have
\[
|\partial_s X| =\begin{cases}
\bigl|X \triangle (X+s)\bigr|, & \text{if } s\neq n/2,\\[4pt]
\bigl|X \triangle (X+s)\bigr|\Big/2, & \text{if } s= n/2.
\end{cases}
\]
Thus,
\[
|\partial X|=\sum_{s\in S^+} |\partial_s X|=
\begin{cases}
\sum\limits_{s\in S^+}\bigl|X \triangle (X+s)\bigr|, & \text{if } n/2\notin S,\\[4pt]
\sum\limits_{s\in S^+\setminus\{n/2\}}\bigl|X \triangle (X+s)\bigr|+\bigl|X \triangle (X+n/2)\bigr|\Big/2, & \text{if } n/2\in S.
\end{cases}
\]

For any \(i\in X\) with \(i+1\notin X\), let \(X'=(X\setminus\{i\})\cup\{i+1\}\). Then
\[
\bigl|X\cap(X+s)\bigr|\leq\bigl| X'\cap(X'+s)\bigr|,
\]
and therefore
\[
\bigl|X' \triangle (X'+s)\bigr|\leq\bigl|X \triangle (X+s)\bigr|.
\]

Then,
\[
|\partial X|\geq|\partial I|=\sum_{s\in S^+} |\partial_s I|=
\begin{cases}
2\sum\limits_{s\in S^+}\bigl|I \setminus (I+s)\bigr|, & \text{if } n/2\notin S,\\[4pt]
2\sum\limits_{s\in S^+\setminus\{n/2\}}\bigl|I \setminus (I+s)\bigr|+\bigl|I \setminus (I+n/2)\bigr|, & \text{if } n/2\in S,
\end{cases}
\]
where \(I\) is an interval of length \(|X|\).

Since $\bigl|I \setminus (I+s)\bigr|=\min\{|I|,s\}$, we have  
\begin{claim}
\label{cl:stronger}
\[
|\partial X|\geq
\begin{cases}
2\sum\limits_{s\in S^+}\min\{|X|,s\}, & \text{if } n/2\notin S,\\[4pt]
2\sum\limits_{s\in S^+\setminus\{n/2\}}\min\{|X|,s\}+\min\{|X|,n/2\}, & \text{if } n/2\in S.
\end{cases}
\]
\end{claim}

\noindent\textbf{Step 2: Establishing a lower bound for the cut in counterexample.}

Suppose, for the sake of contradiction, that for some \(u,v\in U\), there are fewer than \(q\) edge-disjoint \(u\)-\(v\) paths in \(G-F\). By Menger's theorem, there exists an edge cut \(C\subseteq E(G-F)\) separating \(u\) and \(v\) with \(|C|\le q-1\). Choose such a cut of minimum cardinality, and let \(X\) be the side containing \(u\), and \(\overline{X}\) the side containing \(v\). Then \(C=\partial_{G-F}X=\partial_{G-F}\overline{X}\). Since \(\deg_{G-F}(u)=\deg_{G-F}(v)=q\), both \(|X|\) and \(|\overline{X}|\) are at least \(2\). If \(|X|=2\), then
\[
|\partial_{G-F}(X)|\ge q-1+q-2\ge q+1,
\]
a contradiction. If \(|X|=3\), then
\[
|\partial_{G-F}(X)|\ge q-2+2(q-3)\ge q,
\]
again a contradiction. Hence \(|X|\ge4\) and \(|\overline{X}|\ge4\), so \(n=|X|+|\overline{X}|\ge8\). Assume without loss of generality that \(|X|\le|\overline{X}|\); then \(|X|<(n+1)/2\).

Now let \(B=\partial_G X\cap F\) denote the set of crossing edges that belong to the matching. Then \(\partial_G X=C\cup B\) with \(C\cap B=\varnothing\), so \(|\partial_G X|=|C|+|B|\). Define \(F_X=F\cap E(G[X])\), and let \(U_X=U\cap X\) be the set of uncovered vertices in \(X\); clearly \(|U_X|\ge1\), since \(u\in U_X\). Every edge of \(F_X\) covers two vertices of \(X\), every edge of \(B\) covers one vertex of \(X\), and the remaining vertices of \(X\) are exactly \(U_X\). Thus
\[
|X|=|U_X|+2|F_X|+|B|,
\]
and because \(|U_X|\ge1\), we obtain \(|B|\le |X|-1\). Consequently,
\[
|C|=|\partial_G X|-|B|\ge |\partial_G X|-|X|+1. \tag{1}
\]

\noindent\textbf{Step 3: Deriving a contradiction.}

We discuss the cases according to the parity of \(n\), and use Claim~\ref{cl:stronger} and Inequality~(1) to derive a lower bound for \(|C|\).

{\bf Cace A}. $n$ is odd.

\begin{eqnarray*}
|C| &\geq& |\partial_G X|-|X|+1 \\
&\geq &2+2|X|+2\sum\limits_{s\in S^+\setminus\{s_1,l\}}\min\{|X|,s\}-|X|+1\\
&\geq &|X|+2\sum\limits_{s\in S^+\setminus\{s_1,l\}}2+3\\
&\geq &|X|+4(q/2-2)+3\\
&\geq &q+3
\end{eqnarray*}

{\bf Cace B}. $n$ is even.

{\bf Cace B.1}  $|X|\le l$.

\begin{align*}
|C| &\geq 
\begin{cases}
\displaystyle 2 + 2|X| + 2\sum_{s\in S^+\setminus\{s_1,l\}} \min\{|X|,s\} - |X| + 1, & \text{if } n/2\notin S\\[4pt]
\displaystyle 2 + 2|X| + 2\sum_{s\in S^+\setminus\{s_1,l,n/2\}} \min\{|X|,s\} + |X| - |X| + 1, & \text{if } n/2\in S
\end{cases} \\
&\geq 
\begin{cases}
\displaystyle |X| + 4(q/2-2) + 3, & \text{if } n/2\notin S\\[4pt]
\displaystyle 2|X| + 4\big((q-1)/2-2\big) + 3, & \text{if } n/2\in S
\end{cases} \\
&\geq 
\begin{cases}
\displaystyle q + 3, & \text{if } n/2\notin S\\[4pt]
\displaystyle 2q + 1, & \text{if } n/2\in S
\end{cases}
\end{align*}

{\bf Cace B.2}  $|X|- l=1$.

Since $|X|\geq4$, then $l\geq3$.

\begin{align*}
|C| &\geq 
\begin{cases}
\displaystyle 2 + 2l + 2\sum_{s\in S^+\setminus\{1,l\}} \min\{|X|,s\} - |X| + 1, & \text{if } n/2\notin S\\[4pt]
\displaystyle 2 + 2l + 2\sum_{s\in S^+\setminus\{1,l,n/2\}} \min\{|X|,s\} + l - |X| + 1, & \text{if } n/2\in S
\end{cases} \\
&\geq 
\begin{cases}
\displaystyle l + 4(q/2-2) + 2, & \text{if } n/2\notin S\\[4pt]
\displaystyle 2l + 4\big((q-1)/2-2\big) + 2, & \text{if } n/2\in S
\end{cases} \\
&\geq 
\begin{cases}
\displaystyle q + 1, & \text{if } n/2\notin S\\[4pt]
\displaystyle q + 2, & \text{if } n/2\in S
\end{cases}
\end{align*}

{\bf Cace B.3}  $|X|- l=2$.

Since $n \equiv 2 \pmod{4}$, we have $n \ge 10$, and hence $l \ge 3$.

\begin{align*}
|C| &\geq 
\begin{cases}
\displaystyle 2 + 2l + 2\sum_{s\in S^+\setminus\{s_1,l\}} \min\{|X|,s\} - |X| + 1, & \text{if } n/2\notin S\\[4pt]
\displaystyle 2 + 2l + 2\sum_{s\in S^+\setminus\{s_1,l,n/2\}} \min\{|X|,s\} + l - |X| + 1, & \text{if } n/2\in S
\end{cases} \\
&\geq 
\begin{cases}
\displaystyle l + 4(q/2-2) + 1, & \text{if } n/2\notin S\\[4pt]
\displaystyle 2l + 4\big((q-1)/2-2\big) + 1, & \text{if } n/2\in S
\end{cases} \\
&\geq 
\begin{cases}
\displaystyle q, & \text{if } n/2\notin S\\[4pt]
\displaystyle q + 1, & \text{if } n/2\in S
\end{cases}
\end{align*}
In all cases we have $|C| \ge q$, which contradicts the assumption that $|C| \leq q-1$.
\end{proof}

The next lemma further requires the introduction of a technical term. Let
\[
r := 
\begin{cases}
2, & \text{if } n \equiv 0 \pmod{4},\\[2pt]
3, & \text{if } n \equiv 2 \pmod{4} \text{ and } n\neq 10,\\[2pt]
4, & \text{if } n = 10.
\end{cases}
\]

\begin{lemma}\label{ML-1002}
Let \(G=\operatorname{Cir}(n,S)\), where \(q=|S|\), \(n\neq 6\), \(l\in S\), and, if \(n\) is even, \(r\in S\). Suppose \(H\subseteq G\) satisfies \(|E(H)|<n\), \(\Delta(H)\le 2\), and \(H\) is not \(2\)-regular. Then, for every \(t\in\{q-1,q\}\), any two vertices \(u,v\) whose degrees in \(G-E(H)\) are at least \(t\) are joined by \(t\) edge-disjoint \(u\)--\(v\) paths.
\end{lemma}
\begin{proof}
The proof of this lemma follows a similar idea to that of the previous lemma, but the computation requires more care. 

Let  
\[
U=\{u:\deg_{G-E(H)}(u)\ge t\}.
\]
Suppose that for some \(u,v\in U\), there are fewer than \(t-1\) edge-disjoint \(u\)-\(v\) paths in \(G-E(H)\). By Menger's theorem, there exists an edge cut \(C\subseteq E(G-E(H))\) separating \(u\) and \(v\) with \(|C|\le t-2\). Choose such a cut of minimum cardinality, and let \(X\) be the vertex set of the component of \(G-C\) containing \(u\). Set $\overline{X}=V(G)\setminus X$. Then
\[
C=\partial_{G-E(H)}X=\partial_{G-E(H)}\overline{X}.
\]
Since \(\deg_{G-E(H)}(u)\ge t\) and \(\deg_{G-E(H)}(v)\ge t\), we have \(|X|\ge 2\) and \(|\overline{X}|\ge 2\). If \(|X|=2\), then
\[
|\partial_{G-E(H)}(X)|\ge
\begin{cases}
q-1+q-3\ge q, & \text{if } t=q,\\[4pt]
q-2+q-3\ge q-1, & \text{if } t=q-1,
\end{cases}
\]
a contradiction. If \(|X|=3\) and \(q>4\), then
\[
|\partial_{G-E(H)}(X)|\ge
\begin{cases}
q-2+2(q-4)\ge q, & \text{if } t=q,\\[4pt]
q-3+2(q-4)\ge q-1, & \text{if } t=q-1,
\end{cases}
\]
again a contradiction. If \(|X|=3\) and \(q=4\), then since \(n>6\), we have \(l\ge 3\), and therefore, by Claim~\ref{cl:stronger},
\[
|\partial_{G-E(H)}(X)|\ge
\begin{cases}
8-4\ge q, & \text{if } t=q,\\[4pt]
8-5\ge q-1, & \text{if } t=q-1,
\end{cases}
\]
once more a contradiction. Hence \(|X|\ge 4\) and \(|\overline{X}|\ge 4\), so \(n=|X|+|\overline{X}|\ge 8\). 

Now let \(B=\partial_G X\cap H\) be the set of edges of \(H\) crossing the cut. Then
\[
\partial_G X=C\cup B,\qquad C\cap B=\varnothing,
\]
so \(|\partial_G X|=|C|+|B|\). Define \(H_X=H\cap E(G[X])\) and \(U_X=U\cap X\). Clearly \(|U_X|\ge 1\) because \(u\in U_X\).

Each edge of \(H_X\) covers at least one vertex of \(X\), and the endpoints in \(X\) of the edges of \(B\) account for at least \(\lceil |B|/2\rceil\) vertices of \(X\). Therefore,
\[
|X|\ge
\begin{cases}
|U_X|+|H_X|+\lceil |B|/2\rceil, & \text{if } \deg_{G-E(H)}(u)=q,\\[8pt]
|U_X|+|H_X|+\lceil (|B|-1)/2\rceil, & \text{if } \deg_{G-E(H)}(u)=q-1.
\end{cases}
\]
Since \(|U_X|\ge 1\), we obtain
\[
|B|\le
\begin{cases}
2|X|-3, & \text{if } |B| \text{ is odd and } \deg_{G-E(H)}(u)=q,\\[8pt]
2|X|-2, & \text{otherwise}.
\end{cases}
\]
Consequently,
\[
|C|=|\partial_G X|-|B|\ge
\begin{cases}
|\partial_G X|-2|X|+3, & \text{if } |B| \text{ is odd and } \deg_{G-E(H)}(u)=q,\\[8pt]
|\partial_G X|-2|X|+2, & \text{otherwise}.
\end{cases}
\tag{2}
\]

As before, we discuss the cases according to the parity of \(n\). Using Claim~\ref{cl:stronger} and Inequality~(2), we derive a lower bound for \(|C|\).

{\bf Cace A}. $n$ is odd.
\begin{eqnarray*}
|C| &\geq& |\partial_G X|-2|X|+2 \\
&\geq &2+2|X|+2\sum\limits_{s\in S^+\setminus\{s_1,l\}}\min\{|X|,s\}-2|X|+2\\
&\geq &4+4(q/2-2)\\
&\geq &q
\end{eqnarray*}

{\bf Cace B}. $n$ is even.

{\bf Cace B.1}  $|X|\le l$.
\begin{align}
|C| &\geq 
\begin{cases}
\displaystyle 6 + 2\sum_{s\in S^+\setminus\{r,l\}} \min\{|X|,s\}, & \text{if } n/2\notin S \\[4pt]
\displaystyle 6 + 2\sum_{s\in S^+\setminus\{r,l,n/2\}} \min\{|X|,s\}+ |X|, & \text{if } n/2\in S
\end{cases} \\
&\geq 
\begin{cases}
\displaystyle 6 + 2\times
\begin{cases}
0, & \text{if } q=4 \\
1, & \text{if } q=6 \\
3, & \text{if } q=8 \\
2(q/2-2), & \text{if } q\ge 10
\end{cases}, & \text{if } n/2\notin S \\[4pt]
\displaystyle 6+ |X| + 2\times
\begin{cases}
0, & \text{if } q=5 \\
1, & \text{if } q=7 \\
3, & \text{if } q=9 \\
2\big((q-1)/2-2\big), & \text{if } q\ge 11
\end{cases}, & \text{if } n/2\in S
\end{cases} \\
&\geq 
\begin{cases}
\displaystyle q+2, & \text{if } n/2\notin S \\[4pt]
\displaystyle q+5, & \text{if } n/2\in S
\end{cases}
\end{align}
where $r\in\{2,3\}$.

{\bf Cace B.2}  $|X|- l=1$.
\begin{align*}
|C| &\geq 
\begin{cases}
\displaystyle 4 + 2l+2\sum_{s\in S^+\setminus\{2,l\}} \min\{|X|,s\}-2|X|+2, & \text{if } n/2\notin S \\[4pt]
\displaystyle 4 +2l+ 2\sum_{s\in S^+\setminus\{2,l,n/2\}} \min\{|X|,s\}+ l-2|X|+2, & \text{if } n/2\in S
\end{cases} \\
&\geq 
\begin{cases}
\displaystyle 4 + 2\times
\begin{cases}
0, & \text{if } q=4 \\
1, & \text{if } q=6 \\
3, & \text{if } q=8 \\
2(q/2-2), & \text{if } q\ge 10
\end{cases}, & \text{if } n/2\notin S \\[4pt]
\displaystyle 4+ l + 2\times
\begin{cases}
0, & \text{if } q=5 \\
1, & \text{if } q=7 \\
3, & \text{if } q=9 \\
2\big((q-1)/2-2\big), & \text{if } q\ge 11
\end{cases}, & \text{if } n/2\in S
\end{cases} \\
&\geq 
\begin{cases}
\displaystyle q, & \text{if } n/2\notin S \\[4pt]
\displaystyle q+2, & \text{if } n/2\in S
\end{cases}
\end{align*}

{\bf Cace B.3}  $|X|- l=2$.

Since $n \equiv 2 \pmod{4}$, we have $n \ge 10$.

{\bf Cace B.3.1}  $n>10$.

Then $l \ge 5$.
\begin{align*}
|C| &\geq 
\begin{cases}
\displaystyle 6 + 2l+2\sum_{s\in S^+\setminus\{3,l\}} \min\{|X|,s\}-2|X|+2, & \text{if } n/2\notin S \\[4pt]
\displaystyle 6 +2l+ 2\sum_{s\in S^+\setminus\{3,l,n/2\}} \min\{|X|,s\}+ l-2|X|+2, & \text{if } n/2\in S
\end{cases} \\
&\geq 
\begin{cases}
\displaystyle 4 + 2\times
\begin{cases}
0, & \text{if } q=4 \\
1, & \text{if } q=6 \\
3, & \text{if } q=8 \\
2(q/2-2), & \text{if } q\ge 10
\end{cases}, & \text{if } n/2\notin S \\[4pt]
\displaystyle 4+ l + 2\times
\begin{cases}
0, & \text{if } q=5 \\
1, & \text{if } q=7 \\
3, & \text{if } q=9 \\
2\big((q-1)/2-2\big), & \text{if } q\ge 11
\end{cases}, & \text{if } n/2\in S
\end{cases} \\
&\geq 
\begin{cases}
\displaystyle q, & \text{if } n/2\notin S \\[4pt]
\displaystyle q+2, & \text{if } n/2\in S
\end{cases}
\end{align*}

{\bf Cace B.3.2}  $n=10$.
\begin{align*}
|C| &\geq 
\begin{cases}
\displaystyle 8 + 6+2\sum_{s\in S^+\setminus\{3,4\}} s-10+2, & \text{if } 5\notin S \\[4pt]
\displaystyle 8 +6+ 2\sum_{s\in S^+\setminus\{3,4,5\}} s+ 3-10+2, & \text{if } 5\in S
\end{cases} \\
&\geq 
\begin{cases}
\displaystyle 6 + 2\times
\begin{cases}
0, & \text{if } q=4 \\
1, & \text{if } q=6 \\
3, & \text{if } q=8 
\end{cases}, & \text{if } 5\notin S \\[4pt]
\displaystyle 6+ 3 + 2\times
\begin{cases}
0, & \text{if } q=5 \\
1, & \text{if } q=7 \\
3, & \text{if } q=9 
\end{cases}, & \text{if } 5\in S
\end{cases} \\
&\geq 
\begin{cases}
\displaystyle q+2, & \text{if } 5\notin S \\[4pt]
\displaystyle q+4, & \text{if } 5\in S
\end{cases}
\end{align*}
Thus, in every case we obtain \(|C| \ge t\), which contradicts the assumption \(|C| \le t-1\).
\end{proof}

We are now in a position to state the first main result of this paper. Its proof requires the following technical term. Let  
\[
Q_n^{q}:=\operatorname{Cir}(n-1,S)\vee K_1,
\]
that is, \(Q_n^{q}\) is the graph obtained from the circulant graph \(\operatorname{Cir}(n-1,S)\) by adding a new vertex adjacent to every vertex of \(\operatorname{Cir}(n-1,S)\), where \(|S|=q-1\), \(l\in S\), and, if \(n-1\) is even, \(r\in S\). Since $\operatorname{Cir}(n-1,S)$ is a $(q-1)$-cactus, it follows that \(Q_n^{q}\) is a $q$-cactus.

\begin{theorem}\label{MTH-101}
Let $G$ be a simple $q$-cactus of order $n$ consisting of exactly one block. Then $$|E(G)|\leq (q+1)\frac{n-1}{2}.$$
\end{theorem}
\begin{proof}
Obviously, the graph $G$ can be obtained by deleting some edges from a complete graph.
The core idea of this proof is that at least $$\frac{n-1}{2}(n-q-1)=n\frac{n-1}{2}-(q+1)\frac{n-1}{2}$$ edges need to be removed from the complete graph to obtain the graph $G$.

Let \(p = n - q - 1\). Clearly, \(0 \le p \le n-2\). If \(p = 0\) then \(G \cong K_n\); if \(p = n-3\) then \(G \cong C_n\); if \(p = n-2\) then \(G\) is a tree. The statement holds trivially for these extreme values of \(p\), and we therefore restrict our attention to the range \(1 \le p \le n-4\) with \(n \ge 5\). 

If \(p=1\), then \(q=n-2\). In this case, \(G\) is obtained from the \((q+1)\)-cactus \(K_n\) by deleting the edges of a maximum matching. Indeed, only independent edges may be deleted; otherwise, some vertex would lose at least two incident edges, giving degree less than \(q\). Since \(K_n\) contains at most \(\lfloor n/2\rfloor\) pairwise independent edges, at most \(\lfloor n/2\rfloor\) such edges can be deleted. If fewer than \(\lfloor n/2\rfloor\) independent edges are deleted, then at least two vertices retain their full degree \(n-1\), yielding \(n-1=q+1\) edge-disjoint paths between them. Therefore, exactly \(\lfloor n/2\rfloor\) pairwise independent edges—a maximum matching—must be deleted. Deleting a maximum matching from \(K_n\) clearly produces an \((n-2)\)-cactus. Hence
\[
|E(G)|=
\begin{cases}
\displaystyle n(n-1)/2-(n-1)/2=(q+1)(n-1)/2, & \text{if } n \text{ is odd},\\
\displaystyle n(n-1)/2-n/2=(q+1)(n-1)/2-1/2, & \text{if } n \text{ is even},
\end{cases}
\]
and the result follows.

The case \(n=7\) is easily checked, so assume \(n\neq 7\). We first establish the following claim.

\begin{claim}
\label{cl:Edge-maximalq-Cactus}
An edge-maximal \(q\)-cactus of order \(n\) consists of a single block can be constructed as follows:
\begin{itemize}
\item if either \(n\) or \(q\) is odd, it is obtained from a \((q+2)\)-cactus \(Q_n^{q+2}\) by deleting all edges of \(\operatorname{Cir}(n-1,\{i,n-1-i\})\);
\item if both \(n\) and \(q\) are even, it is obtained from a \((q+1)\)-cactus \(Q_n^{q+1}\) by deleting all edges of a perfect matching,
\end{itemize}
where \(i\) is an integer satisfying \(1\le i\le (n-2)/2\) and \(i\notin\{r,l\}\).
\end{claim}

\noindent{\it Proof of Claim}~\ref{cl:Edge-maximalq-Cactus}. We proceed by induction on \(p\) for \(2\le p\le n-4\).

{\bf Base case:} If \(p=2\), then \(q=n-3\), so exactly one of \(n\) and \(q\) is odd.

We show that an edge-maximal \(q\)-cactus of order \(n\) consisting of a single block can be obtained from the \((q+2)\)-cactus
\[
Q_n^{n-1}=\operatorname{Cir}(n-1,\{1,2,\ldots,n-2\})\vee K_1\cong K_n
\]
by deleting all edges of a subgraph \(\operatorname{Cir}(n-1,\{i,n-1-i\})\).

We cannot delete the edges of a subgraph whose maximum degree exceeds \(2\), since this would reduce the degree of some vertex below \(n-3\). If we delete the edges of a \(2\)-regular subgraph \(H\) with fewer than \(n-1\) edges, then at least two vertices retain degree \(n-1\), and hence there are \(n-1\) edge-disjoint paths between them. Therefore assume that \(H\) is not \(2\)-regular and satisfies \(\Delta(H)\le 2\). Since \(Q_n^{n-1}\cong K_n \cong \operatorname{Cir}(n,\{1,2,\ldots,n-1\})\), Lemma~\ref{ML-1002} implies that there exist two vertices connected by at least \(n-2\) edge-disjoint paths.

Thus, at least \(n-1\) edges must be deleted. Since \(Q_n^{n-1}-E\bigl(\operatorname{Cir}(n-1,\{i,n-1-i\})\bigr)\) is a \(q\)-cactus $Q_n^{n-3}$ and is \(2\)-connected, and \(\bigl|E(\operatorname{Cir}(n-1,\{i,n-1-i\}))\bigr|=n-1\), it follows that \(Q_n^{n-1}-E\bigl(\operatorname{Cir}(n-1,\{i,n-1-i\})\bigr)\) is an edge-maximal \(q\)-cactus of order \(n\) consisting of a single block.

{\bf Inductive step:} Assume that the statement holds for \(p=k\), where \(2<k<n-4\), and consider the case \(p=k+1\).

Recall that every edge-maximal \((n-k-2)\)-cactus can be obtained from \(K_n\) by deleting edges. Deleting all required edges at once is equivalent to deleting them sequentially, and the resulting graph does not depend on the order in which the edges are removed. Moreover, deleting a single edge reduces the number of edge-disjoint paths between any two vertices by at most one.

Accordingly, if both \(n\) and \(n-k-2\) are even, we first construct an edge-maximal \((n-k-1)\)-cactus from \(K_n\), and then reduce it to an edge-maximal \((n-k-2)\)-cactus. Otherwise, we first construct an edge-maximal \((n-k)\)-cactus from \(K_n\), and then reduce it to an edge-maximal \((n-k-2)\)-cactus.

{\bf Case A.} Both \(n\) and \(n-k-2\) are even.

By the inductive hypothesis, since \(n-k-1\) is odd, an edge-maximal \((n-k-1)\)-cactus can be obtained from an \((n-k+1)\)-cactus \(Q_n^{n-k+1}\) by deleting all edges of \(\operatorname{Cir}(n-1,\{i,n-1-i\})\). Hence it is a \(Q_n^{n-k-1}\). Let \(v\) be the vertex of degree \(n-1\) in \(Q_n^{n-k-1}\).

We must delete only independent edges whose endpoints both lie in \(V(Q_n^{n-k-1})\setminus\{v\}\); otherwise, some vertex will have degree less than \(n-k-2\). Edges incident to \(v\) are not subject to this restriction. If the edges of a subgraph \(F\) with fewer than \(n/2\) edges are deleted, then \(t\) of the deleted edges are incident to \(v\), and the remaining at most \(n/2-1-t\) edges are independent edges in \(\operatorname{Cir}(n-1,S)\), where \(|S|=n-k-2\) and \(0\le t\le k+1\).

If \(t=0\), then \(v\) still has degree \(n-1\), and at least one vertex \(u\neq v\) has degree \(n-k-2\); hence there are \(n-k-1\) edge-disjoint paths between \(u\) and \(v\). If \(t\ge 1\), then at least \(t+1\) vertices \(u_1,u_2,\dots,u_{t+1}\) in \(V(Q_n^{n-k-1})\setminus\{v\}\) have degree \(n-k-1\). By Lemma~\ref{ML-1001}, for any two of these vertices, say \(u_1\) and \(u_2\), the graph \(\operatorname{Cir}(n-1,S)-E(F)\) contains \(n-k-2\) edge-disjoint paths between them, and there is one additional edge-disjoint path \(u_1vu_2\).

Thus, at least \(n/2\) edges must be deleted. By a slight modification of the proof of Lemma~\ref{ML-1001}, we can show that for any two vertices of \(\operatorname{Cir}(n-1,S)-E(\text{a maximum matching})\), there are exactly \(n-k-3\) edge-disjoint paths between them. Consequently, \(Q_n^{n-k-1} - E(\text{a perfect matching})\) is an \((n-k-2)\)-cactus. Since it is obviously \(2\)-connected, it follows that \(Q_n^{n-k-1} - E(\text{a perfect matching})\) is an edge-maximal \((n-k-2)\)-cactus of order \(n\) consisting of a single block.

{\bf Case B.} Either \(n\) or \(n-k-2\) is odd.

By the inductive hypothesis, since either \(n\) or \(n-k\) is odd, an edge-maximal \((n-k)\)-cactus can be obtained from an \((n-k+2)\)-cactus \(Q_n^{n-k+2}\) by deleting all edges of \(\operatorname{Cir}(n-1,\{i,n-1-i\})\). Hence it is a \(Q_n^{n-k}\). Let \(v\) again be the vertex of degree \(n-1\) in \(Q_n^{n-k}\).

We must delete only edges belonging to a subgraph \(H\) of maximum degree \(2\) whose vertex set is contained in \(V(Q_n^{n-k})\setminus\{v\}\); otherwise, some vertex will have degree less than \(n-k-2\). Edges incident to \(v\) are not subject to this restriction. If we delete the edges of a \(2\)-regular subgraph \(H\) with fewer than \(n-1\) edges, then either at least two vertices in \(V(Q_n^{n-k})\setminus\{v\}\) retain degree \(n-k-1\), in which case Lemma~\ref{ML-1002} gives \(n-k-1\) edge-disjoint paths between them, or at least one vertex \(u\) in \(V(Q_n^{n-k})\setminus\{v\}\) retains degree \(n-k-1\) and the degree of \(v\) remains unchanged, in which case there are \(n-k\) edge-disjoint paths between \(u\) and \(v\). Therefore, assume that \(H\) is not \(2\)-regular and that every vertex of \(H\) other than \(v\) has degree at most \(2\).

If the edges of such a subgraph \(H\) with fewer than \(n-1\) edges are deleted, then \(t\) of the deleted edges are incident to \(v\), and the remaining deleted edges—at most \(n-1-t\) in number—are deleted from \(\operatorname{Cir}(n-1,S)\), where \(|S|=n-k-1\) and \(0\le t\le k+1\).

If \(t=0\), then \(v\) still has degree \(n-1\), and at least one vertex \(u\neq v\) has degree at least \(n-k-2\); hence there are at least \(n-k-1\) edge-disjoint paths between \(u\) and \(v\). If \(t\ge 1\), then either there are at least two distinct vertices \(u_1,u_2\neq v\) of degree \(n-k-1\), in which case Lemma~\ref{ML-1001} yields \(n-k-1\) edge-disjoint paths between them in \(\operatorname{Cir}(n-1,S)-E(H)\), or there are at least two distinct vertices \(u_1,u_2\neq v\) of degree \(n-k-2\), in which case Lemma~\ref{ML-1001} yields \(n-k-2\) edge-disjoint paths between them in \(\operatorname{Cir}(n-1,S)-E(H)\), and the path \(u_1vu_2\) provides one additional edge-disjoint path.

Thus, at least \(n-1\) edges must be deleted. Since \(Q_n^{n-k}-E(\operatorname{Cir}(n-1,\{i,n-1-i\}))\) is a \(Q_n^{n-k-2}\), it is a \(2\)-connected \((n-k-2)\)-cactus; moreover, \(\bigl|E(\operatorname{Cir}(n-1,\{i,n-1-i\}))\bigr|=n-1\). Consequently, \(Q_n^{n-k}-E(\operatorname{Cir}(n-1,\{i,n-1-i\}))\) is an edge-maximal \((n-k-2)\)-cactus of order \(n\) with a single block.

So the claim holds.

By Claim~\ref{cl:Edge-maximalq-Cactus}, an edge-maximal \(q\)-cactus can be constructed from the complete graph \(K_n\) by removing \(w\) edges, where

\[
w = 
\begin{cases}
(n-1)(n-1-q)/2, & \text{ if eihter } n \text{ or } q \text{ is odd},\\
(n-1)\big(n-1-(q+1)\big)/2+n/2, & \text{ if both } n \text{ and } q \text{ are even}.
\end{cases}
\]

Since
\[
\frac{n(n-1)}{2}-w\le (q+1)\frac{n-1}{2},
\]
the theorem follows.

\end{proof}

\begin{theorem}\label{MTH-1011}
Let $G$ be a simple graph of order $n$ and $EP(G) = q(J_n-I_n)$. Then $$\frac{nq}{2}\leq |E(G)|\leq (q+1)\frac{n-1}{2}.$$
\end{theorem}
\begin{proof}
Clearly, every vertex \( u \in V(G) \) satisfies \(\deg(u) \geq q\). By the Handshaking Lemma, we have  
\[
\frac{nq}{2} \leq \frac{1}{2} \sum_{u \in V(G)} \deg(u) = |E(G)|.
\]

Let \( B_1, B_2, \dots, B_s \) be the blocks of \( G \), where each block \( B_i \) has order \( n_i = |V(B_i)| \). By Lemma~\ref{ML-0}, the edge-path matrix of each block \( B_i \) is given by  
\(q(J_{n_i} - I_{n_i})\). Moreover, the orders of the blocks satisfy the relation  
\[
\sum_{i=1}^s n_i = n  + s- 1.
\] 
Obviously 
\[
|E(G)|=\sum_{i=1}^s |E(B_i)|.
\] 

By Theorem~\ref{MTH-101}, we have

\[
|E(G)|\le \sum_{i=1}^s(q+1)\frac{n_i-1}{2}=(q+1)\frac{n-1}{2}.
\] 

\end{proof}

\vskip 0.2cm
We call the edge-maximal $q$-cactus \(Q_n^{q}=\operatorname{Cir}(n-1,S)\vee K_1\) a super $q$-wheel. The proofs of Theorems~\ref{MTH-101}~and~\ref{MTH-1011} immediately yields the following corollary:

\begin{corollary}\label{MC-11}
Let \( G \) be a connected simple graph of order \( n \) such that every block of \( G \) is a super \( q \)-wheel. Then
\[
|E(G)| = \frac{(q+1)(n-1)}{2}.
\]
\end{corollary}

Conjecture~\ref{C-1} is implied by our theorem. 

\begin{corollary}
Let $G$ be a simple graph of order $n$ and $EP(G) \leq q(J_n-I_n)$. Then $$ |E(G)|\leq (q+1)\frac{n-1}{2}.
$$
\end{corollary}

\begin{corollary}
Let $G$ be a simple graph of order $n$ and $EP(G) < q(J_n-I_n)$. Then $$ |E(G)|< (q+1)\frac{n-1}{2}. 
$$
\end{corollary}

An {\it edge-cut} of a graph \( G \) is a subset of edges of the form \( [S, V(G) \setminus S] \), where \( S \) is a nonempty proper subset of \( V(G) \). The {\it edge-connectivity} of \( G \), denoted \( \kappa'(G) \), is the minimum size of an edge-cut in \( G \). A graph \( G \) is \( k \)-{\it edge-connected} if \( \kappa'(G) \geq k \). Furthermore, \( G \) is called {\it minimally} \( k \)-{\it edge-connected} if it is \( k \)-edge-connected, but for every edge \( e \in E(G) \), the graph \( G - e \) is not \( k \)-edge-connected.

\begin{remark}
If \(EP(G) = q(J_n - I_n)\), then obviously \(G\) is minimally \(q\)-edge-connected. Mader~\cite{Mader} proved that a minimally \(q\)-edge-connected simple graph has size at most \(q(n - q)\) when \(n \geq 3q - 2\). For the simple graphs with the edge-path matrix \(q(J_n - I_n)\), the difference between Mader's bound and the bound in Theorem~\ref{MTH-1011} is  given by
\[
q(n - q) - (q+1)\frac{n-1}{2} = \frac{1}{2}(n - 2q - 1)(q - 1).
\]
\end{remark}

\begin{remark}
A minimally \(q\)-edge-connected multigraph (which may contain multiple edges) on \(n\) vertices has at most \(q(n - 1)\) edges, and this bound is best possible for all \(n\) and \(q\)~\cite{Mader}. In a multigraph whose edge-path matrix is \(q(J - I)\), every vertex-minimal block must consist of exactly \(q\) parallel edges. Since there are at most \(n - 1\) such blocks and the total number of edges increases linearly with the number of blocks, it follows that any multigraph with edge-path matrix equal to \(q(J_n - I_n)\) has at most \(q(n - 1)\) edges. This bound is also best possible among all such multigraphs.
\end{remark}

\section{Eulerian graphs and their edge-path matrices }

In this section, we verify Conjecture~\ref{C-2} by first proving a result concerning vertices of odd degree in a simple graph.
\begin{lemma}\label{ML-2222}
Let \(v\) be an odd degree vertex in a simple graph \(G\). Then there is an odd number of edge‑disjoint paths between \(v\) and at least one of its neighbours in \(G\).
\end{lemma}

\begin{proof}
Let $\deg(v) = 2k-1$ for some positive integer $k$.  
If $k=1$, then $v$ has exactly one neighbor, and the edge joining $v$ to this neighbor is the only edge‑disjoint path between them.  

Now assume $k \ge 2$ and let the neighbors of $v$ be $N(v) = \{u_1, \dots, u_{2k-1}\}$.

We argue by contradiction. Suppose, contrary to the statement, that for every neighbor $u_i$ there is an even number of edge‑disjoint paths between $v$ and $u_i$. For each $u_i$, the collection of edge‑disjoint $v$–$u_i$ paths together with the edge $vu_i$ forms an odd number of closed trails. Moreover, every such closed trail contains exactly two neighbors of $v$, one of which is $u_i$.  

Let $A_1, \dots, A_m$ be the family of all $2$-element subsets $\{u_i, u_j\}$ that correspond to the pairs of neighbors appearing together in one of the closed trails described above. Applying the inclusion–exclusion principle, we obtain

\begin{eqnarray*}
2k-1 &=& |A_1\cup\cdots\cup A_m| \\
&=&\sum_{1\leq i_1\leq m} |A_{i_1}|
   - \sum_{1\leq i_1<i_2\leq m} |A_{i_1}\cap A_{i_2}|
   + \cdots \\
&&+ (-1)^{s-1} \sum_{1\leq i_1<\cdots<i_s\leq m} |A_{i_1}\cap\cdots\cap A_{i_s}|
   + \cdots
   + (-1)^{m-1}|A_{1}\cap\cdots\cap A_{m}|.
\end{eqnarray*}

Consider the parity of the right-hand side. The first term is clearly even, so it does not affect the parity. Similarly, any intersection that is empty or contains two elements also leaves the parity unchanged. Hence we only need to examine those intersections that consist of exactly one vertex.

If there are \(s+1\) edge-disjoint paths between \(v\) and \(u_i\), then exactly \(s\) of the sets \(A_1,\dots,A_m\) contain \(u_i\), and these \(s\) sets have intersection \(\{u_i\}\). Consequently, for such a vertex $u_i$, the contribution to the second term is $\binom{s}{2}$, to the third term $\binom{s}{3}$, $\dots$, to the $(s-1)$-th term $\binom{s}{s-1}$, to the $s$-th term $1$, and to all subsequent terms $0$. Since $s$ is odd, we have

\[
-\binom{s}{2} + \cdots + (-1)^{\frac{s-1}{2}-1}\binom{s}{\frac{s-1}{2}} +(-1)^{\frac{s+1}{2}-1}\binom{s}{\frac{s+1}{2}}+ \cdots +\binom{s}{s-2}- \binom{s}{s-1} + \binom{s}{s} \;=\; -s + 1.
\]

Thus the total contribution to the right-hand side from all vertices is even. This forces $2k-1$ to be even—a contradiction.

\end{proof}

We then recall Euler's characterization of Eulerian graphs.

\begin{theorem}[\cite{Euler}]\label{TH-14}
A connected graph $G$ is Eulerian if and only if every vertex of $G$ has even degree.
\end{theorem}

We are now ready to prove Conjecture~\ref{C-2}.

\begin{theorem}\label{MT-44}
A simple graph is Eulerian if and only if every entry of its edge-path matrix is even.
\end{theorem}
\begin{proof}
Assume, for the sake of contradiction, that the edge-path matrix \(EP(G)\) of an Eulerian graph \(G\) contains an odd entry. This means that between the corresponding vertices \(v\) and \(u\), there exist an odd number of edge-disjoint paths. By choosing exactly one edge from each such \(v\)-\(u\) path, we obtain an edge-cut of \(G\) with odd cardinality.

Removing this edge-cut separates \(G\) into two disjoint subgraphs \(G_1\) and \(G_2\), where \(v \in G_1\) and \(u \in G_2\). Because \(G\) is Eulerian, every vertex has even degree. However, an edge-cut of odd size forces each of \(G_1\) and \(G_2\) to contain an odd number of odd-degree vertices—a contradiction, as no graph can have an odd number of odd-degree vertices. 

We prove sufficiency, also by contradiction. Suppose \(G\) is not Eulerian. Then, by Theorem~\ref{TH-14}, \(G\) contains a vertex \(v\) of odd degree. By Lemma~\ref{ML-2222}, there exists a neighbor \(u\) of \(v\) such that there are an odd number of edge-disjoint paths between \(v\) and \(u\). Hence, the entry in the edge-path matrix \(EP(G)\) corresponding to the pair \((v,u)\) is odd.
\end{proof}

The edge-connectivity between two vertices \( v \) and \( u \) is defined as the minimum size of an edge-cut \([S, V(G) \setminus S]\) satisfying \( v \in S \) and \( u \in V(G) \setminus S \). It is well known that the edge connectivity of a vertex pair equals the maximum number of edge-disjoint paths between them. Consequently, Theorem~\ref{MT-44} can be equivalently stated as:
\vskip0.2cm
{\it A simple graph is Eulerian if and only if the edge connectivity between every pair of vertices is even.}
\vskip0.2cm

\section*{Declaration of competing interest}
We declare that we have no financial and personal relationships with other people or organizations that can inappropriately influence our work.

\section*{Acknowledgements}
This work is supported by National Natural Science Foundation of China (Grant Number: 11961070) and Natural Science Foundation of Xinjiang (Grant Number: 2024D01A89).

\section*{Data availability}
No data was used for the research described in the article.



\

\end{document}